\documentclass[11pt,a4paper]{amsart}
\usepackage{mathrsfs}
\usepackage{amsmath}
\usepackage{amsthm}
\usepackage{amsfonts}
\usepackage{amssymb}
\usepackage{latexsym}
\usepackage{geometry}
\usepackage{amscd,amssymb,amsopn,amsmath,amsthm,graphics,amsfonts,mathrsfs,accents,enumerate,verbatim,calc}
\usepackage[dvips]{graphicx}
\usepackage[colorlinks=true,linkcolor=red,citecolor=blue]{hyperref}
\usepackage{color}
\usepackage[all]{xy}

\date{}
\allowdisplaybreaks[4] \footskip=15pt
\renewcommand{\uppercasenonmath}[1]{}

\date{}
\allowdisplaybreaks[4] \footskip=15pt
\renewcommand{\uppercasenonmath}[1]{}

\numberwithin{equation}{section} \theoremstyle{plain}
\newtheorem{lem}{Lemma}[section]
\newtheorem{cor}[lem]{Corollary}
\newtheorem{prop}[lem]{Proposition}
\newtheorem{thm}[lem]{Theorem}

\newtheorem{Defn}[lem]{Definition}
\newtheorem{Ex}[lem]{Example}
\newtheorem{Quest}[lem]{Question}
\newtheorem{Property}[lem]{Property}
\newtheorem{Properties}[lem]{Properties}
\newtheorem{Subprops}{}[lem]
\newtheorem{Para}[lem]{}

\newtheorem{rem}[lem]{Remark}

\newenvironment{exa}{\begin{Ex}\rm}{\end{Ex}}

\newenvironment{para}{\begin{Para}\rm}{\end{Para}}

\newtheorem*{ack*}{ACKNOWLEDGEMENTS}

\newcommand{\pf}{\noindent\begin {proof}}
\newcommand{\epf}{\end{proof}}

\newcommand{\ra}{\rightarrow}
\newcommand{\Ext}{\mbox{\rm Ext}}
\newcommand{\Hom}{\mbox{\rm Hom}}

\newcommand{\im}{\mbox{\rm im}}

\newcommand{\Gpd}{\mbox{\rm Gpd}}

\newcommand{\pd}{\mbox{\rm pd}}

\newcommand{\Add}{\mbox{\rm Add}}
\newcommand{\add}{\mbox{\rm add}}

\begin{document}
\begin{center}
{\large  \bf When the kernel of a complete hereditary cotorsion pair is the additive closure of a tilting module}

\vspace{0.5cm}  Jian Wang$^{a}$, Yunxia Li$^{a}$ and Jiangsheng Hu$^{b}$\footnote{Corresponding author.}\\
$^{a}$College of  Science, Jinling Institute of
Technology, Nanjing 211169, China\\
$^{b}$School of Mathematics and Physics, Jiangsu University of Technology\\
 Changzhou 213001, China\\
E-mail: wangjian@jit.edu.cn, liyunxia@jit.edu.cn and jiangshenghu@jsut.edu.cn
\end{center}


\bigskip
\centerline { \bf  Abstract}

\bigskip
\leftskip10truemm \rightskip10truemm \noindent

In this paper, we study when the kernel of a complete hereditary cotorsion pair is the additive closure of a tilting module. Applications go in three directions. The first is to characterize when the little finitistic dimension is finite. The second is to obtain equivalent formulations for a Wakamatsu tilting module to be a tilting module. The third is to give some new characterizations of Gorenstein rings.
\\
\vbox to 0.3cm{}

\leftskip0truemm \rightskip0truemm
\bigskip
\section { \bf Introduction}
Tilting theory started in the context of finitely generated modules over Artin algebras and was further generalized over arbitrary associative rings with unit and to infinitely generated modules (see \cite{LAHFUCol,LAHFUCol2001,Colby1990,Colby1997,Colby1995,GT}). Recall that the \emph{tilting class} $\mathcal{B}$ associated to a tilting module $T$ over a ring $R$ is the class of $R$-modules satisfying $\mathcal{B}=T^{\bot_\infty}$  \cite{GT}. Tilting modules and classes occur naturally in various areas of contemporary module theory. For example, finiteness of the left little finitistic dimension of a left Noetherian ring $R$ is equivalent to the existence of a particular tilting class (see \cite[Theorem 2.6]{Anegeleri}).

Cotorsion pairs were invented by Salce \cite{SL} in the category of
abelian groups and have been deeply studied in approximation theory of modules \cite{GT}, especially in
the proof of the Flat Cover Conjecture \cite{BEE}. Let $\mathcal{K}_\mathfrak{C}=\mathcal{A}\cap \mathcal{B}$ be the kernel of the cotorsion pair $\mathfrak{C}=(\mathcal{A}, \mathcal{B})$. The additive closure $\textrm{Add}M$ of a module $M$ over a ring $R$ is defined as the class of all modules that
are isomorphic to direct summands of direct sums of copies $M$.
An interesting and deep result in \cite{LAHFUCol} is that an $n$-tilting class $\mathcal{B}$ can be characterized by the properties of cotorsion pairs: $\mathfrak{C}=(\mathcal{A}, \mathcal{B})$ is a complete hereditary cotorsion pair, $\mathcal{A}$ consists of modules of projective dimension at
most $n$ and $\mathcal{K}_\mathfrak{C}$ is closed under arbitrary direct sums. We note that the proof of this result relies on the fact that for any complete hereditary cotorsion pair $\mathfrak{C}=(\mathcal{A}, \mathcal{B})$, if $\mathcal{B}$ is an $n$-tilting class then $\mathcal{K}_\mathfrak{C}$ $=$ $\Add T $ for some $n$-tilting $R$-module $T$. However the converse is not true in general (see Example \ref{em:3.2}).
The following is our first main result which gives some criteria for the kernel of a complete hereditary cotorsion pair to be the additive closure of a tilting module.

\begin{thm}\label{thm: 181.100}
Let $R$ be a ring and $\mathfrak{C}$ $=$ $(\mathcal{A},\mathcal{B})$ a complete hereditary cotorsion pair of $R$-modules, and let ${\mathcal{P}_n}$ {\rm(}${\mathcal{GP}_n}${\rm)} be the class of $R$-modules of finite  projective {\rm(}Gorenstein projective{\rm)} dimension at most $n$. Then the following are equivalent for any nonnegative integer $n$:
\begin{enumerate}
\item $\mathcal{K}_\mathfrak{C}$ $=$ $\Add T$, where $T$ is an $n$-tilting $R$-module.
\item  $\mathcal{K}_\mathfrak{C}\subseteq {\mathcal{P}_n}$, $\mathcal{A}\subseteq {\mathcal{GP}_n}$ and $\mathcal{K}_\mathfrak{C}$ is closed under direct sums.
\item $\mathcal{K}_\mathfrak{C}\subseteq {\mathcal{P}_n}$ and $\mathcal{B}$ $=$ $ T^{\bot_\infty}$ $\cap$ $\mathcal{X}^{\bot_\infty}$, where $T$ is an $n$-tilting $R$-module and $\mathcal{X}$ is a class of strongly Gorenstein projective $R$-modules.

\noindent Moreover, if $\mathcal{B}$ $=$ $G^{\bot_\infty}$ for an $R$-module $G$ and $\Omega^{n} G$ is an $n$-th syzygy of $G$, then the above conditions are  equivalent to
    \item $\mathcal{B}$ $=$ $T^{\bot_{\infty}}\cap N^{\bot_{\infty}}$, where $T$ is an $n$-tilting $R$-module and $N$ is a strongly Gorenstein projective $R$-module.
\item $\mathcal{K}_\mathfrak{C}$ is closed under direct sums and there is a strongly Gorenstein projective $R$-module $M$ in $\mathcal{A}$ such that $\Omega^{n} G$ is a direct summand of $M$.
\item $\mathcal{K}_\mathfrak{C}$ is closed under direct sums and there is a strongly Gorenstein projective $R$-module $N$ in $\mathcal{A}$ such that $(\Omega^{n} G)^{\bot_{\infty}}$ $=$ $ N^{\bot_{\infty}}$.

\end{enumerate}
\end{thm}


Let $\mathcal{P}^{<\infty}$ ($\mathcal{GP}^{<\infty}$) be the class of finitely generated modules with finite projective (Gorenstein projective) dimension.
Recall that the left little finitistic dimension of a ring $R$ is
\begin{center}{findim$(R)$ $=$ $\sup\{\textrm{pd}_{R}M \ | \ M\in{\mathcal{P}^{<\infty}}\}$.}
\end{center}

 As the first application of Theorem \ref{thm: 181.100}, the next result characterizes when the little finitistic dimension is finite. See \ref{5.2} for the proof.

\begin{thm}\label{thm:1706220201} Let $R$ be a left Noetherian ring.  Then the following are equivalent:
\begin{enumerate}

\item {\rm{findim}}$(R)$ $<$ $\infty$.

 \item $\mathcal{K}_\mathfrak{C}$ $=$ $\Add T$, where $\mathfrak{C}=(^\bot((\mathcal{P}^{<\infty})^{\bot_\infty}),(\mathcal{P}^{<\infty})^{\bot_\infty})$ and $T$ is a tilting $R$-module.

\item  $\mathcal{K}_\mathfrak{C}$ $=$ $\Add T$, where $\mathfrak{C}=(^\bot((\mathcal{GP}^{<\infty})^{\bot_\infty}),(\mathcal{GP}^{<\infty})^{\bot_\infty})$ and $T$ is a tilting $R$-module.
    \end{enumerate}
\end{thm}

Note that the equivalence of conditions (1) and (2) in Theorem \ref{thm:1706220201} is due to  Angeleri-H\"{u}gel and Trlifaj, where they proved that for any left Noetherian ring $R$, {\rm{findim}}$(R)$ $<$ $\infty$ if and only if $(\mathcal{P}^{<\infty})^{\bot_\infty}$ is a tilting class (see \cite[Theorem 2.6]{Anegeleri}).

%

The famous Finitistic Dimension Conjecture states that the little
finitistic dimension findim$(R)$ is finite for every Artin algebra $R$ (see \cite{ARS,Bass}). Theorem \ref{thm:1706220201} above gives criteria for the validity of this conjecture.

Wakamatsu in \cite{Wakamatsu} introduced and studied the so-called generalized tilting modules,
which are usually called Wakamatsu tilting modules, see \cite{Beli and Reiten,Mantese}. Tilting modules are Wakamatsu tilting
modules, but the converse is not true  in general  because a Wakamatsu tilting module can have infinite projective dimension.
As the second application of Theorem \ref{thm: 181.100}, we have the next result which gives equivalent formulations for a Wakamatsu tilting module to be a
tilting module. See Theorem \ref{prop2017081202}.

\begin{thm} \label{prop2017081202'} Let $R$ be a ring and $\omega$ a Wakamatsu tilting $R$-module. Fix an exact sequence $0\ra R\ra \omega_0 \xrightarrow{f_0} \cdots \ra \omega_i\xrightarrow{f_i}\cdots$ with $\omega_i$ $\in$ {\rm add$\omega$}
and $\ker (f_i)$ $\in$ $^{\bot_\infty}\omega$ for  $i$ $\geq$ $0$. If we set $A$ $=$ ${\bigoplus\limits_{i\geqslant0}}\ker (f_i)$, then the following are equivalent for any nonnegative integer $n$:
\begin{enumerate}
\item $\omega$ is an $n$-tilting $R$-module.
\item $\omega^{\bot_> n}$ $=$ $A^{\bot_> n}$ and $\mathcal{K}_{\mathfrak{C}}$ $=$ $\Add T$, where $\mathfrak{C}$ $=$ $(^\bot(A^{\bot_\infty}),A^{\bot_\infty})$ and $T$ is an $n$-tilting $R$-module.

\item $\mathcal{K}_{\mathfrak{C}}$ $=$ $\Add T$, where $\mathfrak{C}$ $=$ $(^\bot((\omega\oplus A)^{\bot_\infty}),(\omega\oplus A)^{\bot_\infty})$ and $T$ is an $n$-tilting $R$-module.

\end{enumerate}
\end{thm}
As a consequence of Theorem \ref{prop2017081202'}, we characterize when a Wakamatsu tilting module of
finite projective dimension is a tilting module, see Corollary \ref{cor:2018081601}.

%
%
%
%

Recall that a ring $R$ is called \emph{Gorenstein} (or \emph{Iwanaga-Gorenstein}) \cite{Iwer} if it is both left and right Noetherian and $R$ has finite
self-injective dimension on either side. In the case of commutative rings, this definition of
Gorenstein rings coincides with Gorenstein rings of finite Krull dimension originally defined by Bass in \cite{Bass1963}. For more details about Gorenstein rings, see \cite[Section 3]{AM2002} and \cite[Chapter 9]{EJ}.

As the third application of Theorem \ref{thm: 181.100}, the following result gives some new characterizations of Gorenstein rings. See Propositions \ref{prop20170806001} and \ref{prop20170509001}.

\begin{thm} \label{thm:1.4} Let $R$ be a commutative ring.  Then the following are equivalent:
\begin{enumerate}

\item $R$ is a Gorenstein ring.

 \item $\mathcal{K}_\mathfrak{C}$ $=$ $\Add T$, where $\mathfrak{C}$ $=$ $(^\bot\mathcal{GI}$, $\mathcal{GI})$ and $T$ is a tilting $R$-module.

\item  $\mathcal{K}_\mathfrak{C}$ $=$ $\Add T$, where $\mathfrak{C}$ $=$ $(R$-{\rm Mod}, $\mathcal{I})$ and $T$ is a tilting $R$-module.

\noindent Moreover, if  $R$ is a commutative Noetherian ring of finite Krull dimension, then the above conditions are  equivalent to

 \item For any exact sequence  $ \cdots\ra T_2 \xrightarrow{d_2} T_1 \xrightarrow{d_1}T_0 \xrightarrow{d_0} \cdots$ of tilting $R$-modules, $\mathcal{K}_{\mathfrak{C}}$ $=$ $\Add T$, where $\mathfrak{C}$ $=$ $(^\bot((\oplus\ker(d_i))^{\bot_\infty}),(\oplus\ker(d_i))^{\bot_\infty})$  and $T$ is a tilting $R$-module.

 \item For any exact sequence  $ \cdots\ra T_2 \xrightarrow{d_2} T_1 \xrightarrow{d_1}T_0 \xrightarrow{d_0} \cdots$ of tilting $R$-modules, each $\ker(d_i)$ has finite Gorenstein projective dimension.
 \end{enumerate}

\end{thm}
We note that the equivalence of conditions (1) and (2) in Theorem \ref{thm:1.4} is related to a result by Angeleri-H\"{u}gel, Herbera and Trlifaj \cite{LDT Til Gorenstein}. In their paper, they proved that a two-sided Noetherian ring $R$ is Gorenstein if and only if Gorenstein injective left and right $R$-modules form a tilting class, see \cite[Theorem 3.4]{LDT Til Gorenstein}.

We conclude this section by summarizing the contents of this paper. Section 2 contains some notations, definitions and lemmas for use throughout this paper. Section 3 is devoted to proving Theorem \ref{thm: 181.100}.
Section 4 is some applications of Theorem \ref{thm: 181.100}, including the proofs of Theorems \ref{thm:1706220201}-\ref{thm:1.4}.

\section {\bf Preliminaries}
Throughout this paper, $R$ is an associative ring with identity and $R$-Mod is the category of left $R$-modules. Unless otherwise stated, all $R$-modules are left $R$-modules.

Next we recall some basic definitions and properties needed in the sequel. For more details the reader can consult \cite{Anegeleri,EJ,GT}.

\vspace{2mm}
\noindent{\bf Notation.} Let $\mathcal{C}$ be a full subcategory of $R$-Mod and $n$ a nonnegative integer. The classes $\mathcal{C}^\bot$, $^\bot\mathcal{C}$, $\mathcal{C}^{\bot_{> n}}$, $\mathcal{C}^{\bot_{\infty}}$ and $^{\bot_{\infty}}\mathcal{C}$ of $\mathcal{C}$ are defined as follows:
$$\mathcal{C}^\bot=\{M\in R{\text-}{\rm Mod}\ |\ \Ext_{R}^1(C, M)=0\ for\ all\ C\in\mathcal{C}\},$$
$$^\bot\mathcal{C}=\{M\in R{\text-}{\rm Mod}\ |\ \Ext_{R}^1(M, C)=0\ for\ all\ C\in\mathcal{C}\},$$
$$\mathcal{C}^{\bot_{> n}}=\{M\in R{\text-}{\rm Mod}\ |\ \Ext_{R}^{n+i}(C, M)=0\ for\ all\ C\in\mathcal{C}\ and \ all\ i\geq1 \},$$
$$\mathcal{C}^{\bot_{\infty}}=\{M\in R{\text-}{\rm Mod}\ |\ \Ext_{R}^i(C, M)=0\ for\ all\ C\in\mathcal{C}\ and\ all\ i\geq1 \}.$$
$$^{\bot_{\infty}}\mathcal{C}=\{M\in R{\text-}{\rm Mod}\ |\ \Ext_{R}^i(M, C)=0\ for\ all\ C\in\mathcal{C}\ and\ all\ i\geq1 \}.$$
For $\mathcal{C}$ $=$ $\{$$C$$\}$, we write for short $C^\bot$, $^\bot C$, $C^{\bot_{> n}}$, $C^{\bot_{\infty}}$ and $^{\bot_{\infty}}C$ instead of $\{$$C$$\}^\bot$, $^\bot\{$$C$$\}$, $\{C\}^{\bot_{> n}}$, $\{$$C$\}$^{\bot_{\infty}}$ and $^{\bot_{\infty}}$$\{$$C$\}, respectively.

 If $\xymatrix{\cdots\ar[r]&P_{i}\ar[r]^{f_{i}}&\cdots\ar[r]&P_{0}\ar[r]^{f_{0}}&M\ar[r]&0}$ is a projective resolution of $M$, then $\im{(f_{i})}$ is called an \emph{$i$-th syzygy} of $M$. We denote it by $\Omega^{i}M$ $($$\Omega^{0}M$ $=$ $M$$)$.

Let $n$ be a nonnegative integer. For any $R$-module $M$, $\pd_{R}M$ is the projective dimension of $M$. For convenience, we set $\mathcal{P}_n$ the class of $R$-modules $M$ with $\pd_{R}M$ $\leq n$.

Given an $R$-module $M$, we denote by $\Add M$ (resp. $\add M$) the class of all modules that
are isomorphic to direct summands of direct sums (resp. finite direct sums) of copies $M$.

%

\vspace{2mm}
\noindent{\bf  Cotorsion pairs.}  Let $\mathcal{A}$ and $\mathcal{B}$ be classes in $R$-Mod.
  Recall that a pair $\mathfrak{C}=(\mathcal{A},\mathcal{B})$ is called a \emph{cotorsion pair} \cite{GT,SL} if $\mathcal{A}$ $=$ $^\bot\mathcal{B}$ and $\mathcal{B}$ $=$ $\mathcal{A}^\bot$. The class $\mathcal{K}_\mathfrak{C}=\mathcal{A}\cap \mathcal{B}$ is called the kernel of $\mathfrak{C}$.

   A cotorsion pair ($\mathcal{A}$, $\mathcal{B}$) is said to be
\emph{hereditary} \cite{GT} if $\Ext_{R}^{i}(A,B)$ $=$ $0$ for all $i$ $\geq$ $1$, $A$ $\in$ $\mathcal{A}$ and $B$ $\in$ $\mathcal{B}$, equivalently, if
whenever $0\ra A_{3}\ra A_{2} \ra A_{1}\ra 0$ is exact with $A_{2}$, $A_{1}$ $\in$ $\mathcal{A}$, then $A_{3}$ is also in $\mathcal{A}$, or equivalently, if whenever $0\ra B_{1}\ra B_{2} \ra B_{3}\ra 0$ is exact with $B_{1}$, $B_{2}$ $\in$ $\mathcal{B}$, then $B_{3}$ is also in $\mathcal{B}$.

 A cotorsion pair $(\mathcal{A},\mathcal{B})$ is \emph{complete} \cite{GT} if one of the following two equivalent conditions holds:
\begin{itemize}
\item For each $R$-module $M$, there is an exact sequence $0\ra M\ra B\ra L\ra 0$ with $B$ $\in$ $\mathcal{B}$ and $L$ $\in$ $\mathcal{A}$.
\item For each $R$-module $M$, there is an exact sequence $0\ra D\ra C\ra M\ra 0$ with $C$ $\in$ $\mathcal{A}$ and $D$ $\in$ $\mathcal{B}$.
\end{itemize}

A cotorsion
pair $(\mathcal{A},\mathcal{B})$  is \emph{generated by a set} \cite{GT} provided that there is a set $\mathcal{S}$ of $R$-modules such that $\mathcal{S}^\bot=\mathcal{B}$ $($i.e., $(\mathcal{A},\mathcal{B})$ $=$ $(^\bot(\mathcal{S}^\bot), \mathcal{S}^\bot)$$)$. In the literature, this is sometimes called the cotorsion pair\emph{ cogenerated} by $S$. Here, however, we use the terminology from \cite{GT}. By \cite[Theorem 3.2.1]{GT}, each cotorsion pair generated by a set is complete. Moreover, we have the following lemma.

\begin{lem} \label{lem: 20170819001} Let $R$ be a ring and $M$ an $R$-module. Then $(^\bot(M^{\bot_{\infty}}), M^{\bot_{\infty}})$ is a complete hereditary cotorsion pair.
\end{lem}
\begin{proof} Let $M$ be an $R$-module. It is easy to check that $M^{\bot_{\infty}}$ $=$ $U^{\bot}$, where $U$ $=$ ${\bigoplus\limits_{i\geqslant0}}$$\Omega^{i}M$. Thus $(^{\bot}(M^{\bot_{\infty}}), M^{\bot_{\infty}})$ is a complete
cotorsion pair by \cite[Theorem 3.2.1]{GT}.  One can check that $(^\bot(M^{\bot_{\infty}}), M^{\bot_{\infty}})$ is also a hereditary cotorsion pair by the definition.
\end{proof}

  In the following, for an $R$-module $M$, we set $\mathcal{K}_{M}={^\bot(M^{\bot_{\infty}})}\cap M^{\bot_{\infty}}$. It follows from
  Lemma \ref{lem: 20170819001} that $\mathcal{K}_{M}=\mathcal{K}_\mathfrak{C}$ whenever $\mathfrak{C}=(^\bot(M^{\bot_{\infty}}), M^{\bot_{\infty}})$.
For more detailed information about cotorsion pairs, we refer the reader to \cite{EJ,GT}.

\vspace{2mm}
\noindent{\bf Gorenstein projective modules.} Following \cite{EJ GP,H}, an $R$-module $G$ is called \emph{Gorenstein projective} if there is an  exact sequence of projective $R$-modules $$\mathbf{P}:\xymatrix{\cdots\ar[r]& P^{-2}\ar[r]^{f^{-2}}&  P^{-1}\ar[r]^{f^{-1}} & P^{0}\ar[r]^{f^{0}}& P^{1}\ar[r]^{f^{1}}&\cdots}$$ such that $G$ $\cong$ $\ker(f^{0})$ and $\Hom_R(\mathbf{P},Q)$ is exact for every projective $R$-module $Q$. In this case, the complex $\mathbf{P}$ is also called a totally acyclic complex of projective $R$-modules. It is clear that each $\ker(f^{i})$ is Gorenstein projective.

Let $n$ be a nonnegative integer. The \emph{Gorenstein projective dimension}, $\Gpd_{R}G$, of an $R$-module $G$ is defined by declaring that $\Gpd_{R}G$ $\leq$ $n$ if, and only if there is an exact sequence $0\ra G_{n}\ra \cdots \ra G_{0}\ra G\ra 0$ with all $G_{i}$ Gorenstein projective (see \cite[Definition 2.8]{H}).
By \cite[Proposition 2.27]{H}, $\Gpd_R M$ $=$ $\pd_R M$ whenever $\pd_R M<\infty$, and so any Gorenstein projective module with finite projective dimension is projective.

In the following, we use $\mathcal{GP}_n$ to denote  the class of $R$-modules $M$ with $\Gpd_{R}M$ $\leq n$. By \cite[Proposition 2.7]{H}, we get that an $R$-module $G$ belongs to $\mathcal{GP}_n$ if and only if any $n$-th syzygy $\Omega^{n}G$ of $G$ is Gorenstein projective.

\begin{lem} \label{lem: 1.09} Let $G$ and $M$ be $R$-modules. Then the following are true for any nonnegative integer $n$:
\begin{enumerate}
\item If $\Ext_{R}^{i}(G,M)$ $=$ $0$ for all $i$ $\geq$ $n+1$, then $\Ext^{i}_{R}(W,M)$ $=$ $0$ for all $W$ $\in$ $^\bot(G^{\bot_{\infty}})$ and $i$ $\geq$ $n+1$.

\item If $G$ $\in$ $\mathcal{P}_n$, then $^\bot(G^{\bot_\infty})$ $\subseteq$ $\mathcal{P}_n$.
\item If $G$ $\in$ $\mathcal{GP}_n$, then $^\bot(G^{\bot_\infty})$ $\subseteq$ $\mathcal{GP}_n$.
\end{enumerate}
\end{lem}

\begin{proof}$(1)$ Consider the exact sequence $0\ra M \ra E_{0} \ra \cdots \ra E_{n-1} \ra L\ra 0$ with each $E_{i}$ injective, we have $\Ext^{k}_{R}(G,L)$ $\cong$ $\Ext^{n+k}_{R}(G,M)$ for $k$ $\geq$ $1$.  By assumption, $\Ext^{k}_{R}(G,L)$ $=$ $0$ for $k$ $\geq$ $1$. Thus $L$ $\in$ $G^{\bot_\infty}$. Note that $(^\bot(G^{\bot_\infty}),G^{\bot_\infty})$ is a hereditary cotorsion pair.   Therefore $\Ext^{k}_{R}(W,L)$ $=$ $0$ for any $W$ $\in$ $^\bot(G^{\bot_{\infty}})$ and $k$ $\geq$ $1$. So $\Ext^{n+k}_{R}(W,M)$ $=$ $0$ for $k$ $\geq$ $1$ by noting that $\Ext^{k}_{R}(W,L)$ $\cong$ $\Ext^{n+k}_{R}(W,M)$, as desired.

$(2)$ Since $G$ $\in$ $\mathcal{P}_n$, we get that $^\bot(G^{\bot_\infty})$ $\subseteq$ $^\bot((\mathcal{P}_n)^{\bot_\infty})$. By \cite[Theorem 4.1.12]{GT}, $(\mathcal{P}_n, (\mathcal{P}_n)^\bot)$ is a hereditary cotorsion pair. So $^\bot((\mathcal{P}_n)^{\bot_\infty})$ $=$ $\mathcal{P}_n$.

$(3)$  Let $U$ $=$ ${\bigoplus\limits_{i\geqslant0}}$$\Omega^{i}G$. It is clear that $G^{\bot_{\infty}}$ $=$ $U^\bot$. If $n=0$ and $G$ is Gorenstein projective, then $U$ is also Gorenstein projective.  So the result holds by \cite[Theorem 3.2]{EnochGptrans} and \cite[Corollary 3.2.4]{GT}.

For $n\geq1$, we assume that $\Gpd_{R}G$ $\leq$ $n$.
Then $\Omega^{n}G$ is Gorenstein projective, and so every $R$-module in  $^\bot(({\Omega^{n} G})^{\bot_\infty})$ is Gorenstein projective. Let $L$  be an $R$-module in  $({\Omega^{n} G})^{\bot_\infty}$.
Assume that $N$ is an $R$-module in $^\bot(G^{\bot_\infty})$.  Thus $\Ext_R^{n+1}(N,L)=0$ by (1), and hence $ \Ext_R^{1}(\Omega^{n}N,L) $ $\cong$ $ \Ext_R^{n+1}(N,L)$ $=$ $0$. It follows that $\Omega^{n}N$ $\in$ $^\bot(({\Omega^{n} G})^{\bot_\infty})$ and then $\Omega^{n}N$ is Gorenstein projective. So $\Gpd_{R}N$ $\leq$ $n$.  This completes the proof.
\end{proof}

\vspace{2mm}
\noindent{\bf Strongly Gorenstein projective modules.}
Recall that an $R$-module $M$ is called \emph{strongly Gorenstein projective} \cite{DB} if there is an exact sequence of projective $R$-modules $$\mathbf{P}:\xymatrix{\cdots\ar[r]& P\ar[r]^{f}&  P\ar[r]^{f} & P\ar[r]^{f}&  P\ar[r]^{f}&\cdots}$$ such that  $M$ $\cong$ $\ker(f)$ and $\Hom_R(\mathbf{P},Q)$ is exact for every projective $R$-module $Q$.

All projective $R$-modules are strongly Gorenstein projective and the class of strongly Gorenstein projective modules is closed under direct sums. The principal role of the strongly
Gorenstein projective modules is to give the following characterization of Gorenstein projective
modules \cite[Theorem 2.7]{DB}: an $R$-module is Gorenstein projective if and only if it is a direct summand of a strongly Gorenstein projective $R$-module. Moreover, a careful reading
of the proof of \cite[Theorem 2.7]{DB} gives the following lemma.

\begin{lem} \label{lem: 2017082101'}If $\cdots \ra P^{-1}\xrightarrow{f^{-1}}P^{0}\xrightarrow{f^{0}}P^{1}\xrightarrow{f^{1}}  \cdots$ is an exact sequence of projective $R$-modules with all $\ker(f^i)$  Gorenstein projective, then  $\oplus\ker(f^i)$ is strongly Gorenstein projective.
\end{lem}Recall that a full subcategory $\mathcal{C}$ of $R$-Mod is \emph{thick} \cite{H S app C P} if $\mathcal{C}$ is closed under direct summands and has the two out of three property: for every exact sequence of $R$-modules $0\rightarrow A\rightarrow B\rightarrow C \rightarrow 0$ with two terms in $\mathcal{C}$, then the third one is also in $\mathcal{C}$.

Following \cite{AB CM}, a complete hereditary cotorsion pair $(\mathcal{A},\mathcal{B})$ is said to be \emph{projective} if $\mathcal{A}\cap\mathcal{B}$ is the class of projective $R$-modules.

\begin{lem} \label{lem: 1.6}$($\cite[Lemma 2.1]{JwandliLiang}$)$ The following are true for any strongly Gorenstein projective $R$-module $N$:

\begin{enumerate}
\item  $N^{\bot}$ is a thick subcategory of {\rm $R$-Mod}.
\item  $(^\bot{(N^\bot)},N^\bot)$ is a projective cotorsion pair.
\end{enumerate}
\end{lem}

\vspace{2mm}
\noindent{\bf  Tilting modules.}  Let us recall some basic facts about (not necessarily finitely generated) tilting
modules. An $R$-module $T$ is \emph{tilting} \cite{LAHFUCol,Colby1995} provided that the following hold:
\begin{enumerate}
\item[$(T1)$] $\pd_RT$ $<$ $\infty$.
\item[$(T2)$] $\Ext_{R}^{i}(T,T^{(\lambda)})$ $=$ $0$ for each $i$ $\geq$ $1$ and for every cardinal $\lambda$.
\item[$(T3)$] There is a long exact sequence $0\ra R\ra T_{0}\ra \cdots\ra T_{r}\ra 0$ with $T_{i}$ $\in$ $\Add T$ for $0$ $\leq$ $i$ $\leq$ $r$, where $r$ is the projective dimension of $T$.
\end{enumerate}


The class $T^{\bot_{\infty}}$ is called the \emph{tilting class} induced by $T$. Further, $T$ and $T^{\bot_{\infty}}$ are
called \emph{$n$-tilting} when $T\in{\mathcal{P}_n}$. An $n$-tilting class $\mathcal{B}$ can be characterized by the
properties: $\mathcal{B}$ is closed under direct sums, $(^{\bot}\mathcal{B}, \mathcal{B})$ is a complete hereditary cotorsion pair and $^{\bot}\mathcal{B}$ $\subseteq\mathcal{P}_n$. By the proof of  \cite[Theorem 4.1]{LAHFUCol}, we have the following lemma.

\begin{lem} \label{lem: 2.5''} Let $\mathfrak{C}=(\mathcal{A}, \mathcal{B})$ be a complete hereditary cotorsion pair over a ring $R$. If $\mathcal{B}$ is an $n$-tilting class, then $\mathcal{K}_\mathfrak{C}$ $=$ $\Add T $ for a tilting $R$-module $T$.
\end{lem}

We end this section with the following example which shows that the converse of Lemma \ref{lem: 2.5''}  is not true in general.

\begin{exa}\label{em:3.2} Let $R$ be a ring and $N$ a strongly Gorenstein projective $R$-module but not projective $($see \cite[Example 2.5]{DB}$)$. Then $\mathfrak{C}=(^\bot{(N^\bot)},N^\bot)$ is a projective cotorsion pair by Lemma \ref{lem: 1.6}. Thus $\mathcal{K}_\mathfrak{C}$ $=$ $\Add R $. We claim that $N^{\bot}$ is not a tilting class. Indeed, if $N^{\bot}$ is a tilting class, then $\pd_R N$ $<$ $\infty$ by \cite[Theorem 4.1]{LAHFUCol}. Hence $N$ is projective. This is a contradiction.
\end{exa}

\section{\bf Proof of Theorem \ref{thm: 181.100}}\label{proof}

In the following part, we will prove our main theorem. For this purpose, we need some technical results.

One can check that an $R$-module $M$ is Gorenstein projective if and only if there is an exact
sequence $0\ra M\ra Q^0\xrightarrow{g^0} \cdots \ra Q^n\xrightarrow{g^n}\cdots$ such that $Q^j$ is projective and $\Ext^i_R(\ker(g^j),L)$ $=$ $0$ for every
 projective $R$-module $L$, $j$ $\geq$ $0$ and $i$ $\geq$ $1$. Moreover, we have the following lemma.

\begin{lem} \label{lem: 1.8} Let $G$ be an $R$-module and $n$ a nonnegative integer. The following are equivalent:
\begin{enumerate}
\item $\Gpd_{R}G$ $\leq$ $n$.
\item There is an exact sequence of $R$-modules $$\xymatrix{0\ar[r]&G\ar[r]&P^{0}\ar[r]^{f^{0}}&\cdots\ar[r]&P^{m}\ar[r]^{f^{m}}&\cdots}$$ with $\pd_{R}P^{j}$ $\leq$ $n$ such that $\Ext_{R}^{i}(\ker(f^{j}),P)$ $=$ $0$ for every projective $R$-module $P$, $j$ $\geq$ $0$ and $i$ $\geq$ $n+1$.

\end{enumerate}
\end{lem}
\begin{proof}
$(1)$ $\Rightarrow$ $(2)$. By \cite[Lemma 2.17]{LWchris}, there is an exact sequence $0\ra G\ra D\xrightarrow{\alpha} G^{0}\ra 0$ of $R$-modules with $G^{0}$ Gorenstein projective and $\pd_{R}D$ $\leq$ $n$. Thus there is an exact sequence of $R$-modules $$\xymatrix{0\ar[r]&G^{0}\ar[r]^{\beta}&P^{1}\ar[r]^{f^{1}}&\cdots\ar[r]&P^{m}\ar[r]^{f^{m}}&\cdots}$$ with each $P^{j}$ projective such that $\Ext_{R}^{i}(\ker(f^{j}),P)$ $=$ $0$ for every projective $R$-module $P$, $j$ $\geq$ $1$ and $i$ $\geq$ $1$. So we have the exact sequence of $R$-modules
$$\xymatrix{0\ar[r]&G\ar[r]&P^{0}\ar[r]^{f^{0}}&P^{1}\ar[r]^{f^{1}}& \cdots \ar[r]&P^{m}\ar[r]^{f^{m}}&\cdots}$$ with $f^{0}=\beta\alpha$ and $P^0$ $=$ $D$. By
\cite[Theorem 2.20]{H}, $\Ext_{R}^{i}(G,P)$ $=$ $0$ for every projective $R$-module $P$ and $i$ $\geq$ $n+1$, as desired.

$(2)$ $\Rightarrow$ $(1)$. To prove that $\Gpd_{R}G$ $\leq$ $n$, it is sufficient to prove that $\Omega^nG$ is Gorenstein projective. Consider the exact sequence $0\ra \ker(f^{m}) \ra P^{m}\ra \ker(f^{m+1})\ra 0$ for all $m$ $\geq$ $0$,
we can construct the following commutative diagram as in \cite[Lemma 8.2.1]{EJ}
$$\xymatrix{&0&0&0&             \\
0\ar[r]&\ker(f^{m})\ar[r]\ar[u]&P^{m}\ar[r]\ar[u]&\ker(f^{m+1})\ar[r]\ar[u]& 0\\
           0\ar[r]&Q^{m,0}\ar[r]\ar[u]&Q^{m,0}\oplus Q^{m+1,0}\ar[r]\ar[u]&Q^{m+1,0}\ar[r]\ar[u]& 0 \\
&\vdots\ar[u]&\vdots\ar[u]&\vdots\ar[u]&   \\
0\ar[r]&\Omega^{n}\ker(f^{m})\ar[r]\ar[u]&\Omega^{n}P^{m}\ar[r]\ar[u]&\Omega^{n}\ker(f^{m+1})\ar[r]\ar[u]& 0 \\
&0\ar[u]&0\ar[u]&0,\ar[u]& }$$  where rows and columns are exact and each $Q^{t,k}$ is projective.  Since $\pd_{R}P^{m}$ $\leq$ $n$, $\Omega^{n}P^{m}$ is projective for all $m$ $\geq$ $0$. Note that $\Omega^{n}G$ $=$ $\Omega^{n}\ker(f^{0})$. Thus we have the following commutative diagram with exact row
$$\xymatrix@C=10pt@R=13pt{0\ar[r]&\Omega^{n}G\ar[r]&\Omega^{n}P^{0}\ar[r]&\cdots\ar[r]
&\Omega^{n}P^{m-1}\ar[rd]\ar[rr]&&\Omega^{n}P^{m}\ar[r]&\cdots\\
                               &&&&&\Omega^{n}\ker(f^{m})\ar[ru]\ar[rd]&&
\\&&&&0\ar[ru]&&0.&}$$ By assumption, one can check that $\Ext_{R}^{j}(\Omega^{n}\ker(f^{m}),P)$ $\cong$ $\Ext_{R}^{n+j}(\ker(f^{m}),P)$ $=$ $0$ for every projective $R$-module $P$, $j$ $\geq$ $1$ and $m$ $\geq$ $0$. Hence $\Omega^{n}G$ is Gorenstein projective, as desired.
This completes the proof.
\end{proof}

\begin{lem} \label{lem: 1803261.1} Let $R$ be a ring and $\mathfrak{C}=(\mathcal{A},\mathcal{B})$ a complete hereditary cotorsion pair of $R$-modules. If $\mathcal{K}_\mathfrak{C}\subseteq {\mathcal{P}_n}$, $\mathcal{A}$ $\subseteq$ ${\mathcal{GP}_n}$ and $G$ is an $R$-module in $\mathcal{A}$, then there is a strongly Gorenstein projective $R$-module $N$ in $\mathcal{A}$ such that $\Omega^n{G}$ is a direct summand of $N$.
\end{lem}
\begin{proof} Let $G$ be an $R$-module in $\mathcal{A}$. Note that $(\mathcal{A},\mathcal{B})$ is a complete cotorsion pair. Then there is an exact sequence of $R$-modules $$\xymatrix{0\ar[r]&G\ar[r]&P^{0}\ar[r]^{f^{0}}&\cdots\ar[r]&P^{m}\ar[r]^{f^{m}}&\cdots}$$ with $P^{j}$ $\in$ $\mathcal{K}_\mathfrak{C}$ and $\ker(f^{j})$ $\in$ $\mathcal{A}$ for $j\geq0$. By assumption, $\mathcal{A}$ $\subseteq $ ${\mathcal{GP}_n}$. Thus $\Gpd_R\ker(f^j) $ $\leq$ $n$ for $j$ $\geq$ $0$. It follows that $\Ext_{R}^{n+i}(\ker(f^{j}),P)$ $=$ $0$ for any projective $R$-module $P$, $j$ $\geq$ $0$ and  $i$ $\geq$ $1$ by \cite[Theorem 2.20]{H}. Using similar arguments as in the proof of $(2)$ $\Rightarrow$ $(1)$ in Lemma \ref{lem: 1.8}, we get an exact sequence of $R$-modules

 $$\xymatrix@C=10pt@R=13pt{0\ar[r]&\Omega^{n}G\ar[r]&\Omega^{n}P^{0}\ar[r]^{\indent g^0}&\cdots\ar[r]
&\Omega^{n}P^{m-1}\ar[rd]\ar[rr]^{g^{m-1}}&&\Omega^{n}P^{m}\ar[r]^{\indent g^m}&\cdots\\
                               &&&&&\Omega^{n}\ker(f^{m})\ar[ru]\ar[rd]&&
\\&&&&0\ar[ru]&&0&}$$ where  $\Omega^{n}P^{j}$ is projective and  $\ker(g^j)$ $=$ $\Omega^n\ker(f^j)$ is Gorenstein  projective for  $j$ $\geq$ $0$. Since each $\ker{(f^j)}$ $\in$ $\mathcal{A}$ and $(\mathcal{A},\mathcal{B})$ is a hereditary cotorsion pair, $\Omega^n\ker{(f^j)}$ $\in$ $\mathcal{A}$ for $j$ $\geq$ $0$.
 So we obtain an exact sequence of $R$-modules $$\mathbf{P_+}:0\ra \Omega^{n}G \ra \Omega^{n}P^{0}\xrightarrow{g^0} \cdots \ra \Omega^{n}P^{m}\xrightarrow{g^m} \cdots,$$ where  $\ker(g^j)$ is  Gorenstein projective and belonging to $\mathcal{A}$ for $j$ $\geq$ $0$.  On the other hand, we notice that $\Omega^n G$ is  Gorenstein projective and belonging to $\mathcal{A}$, thus we have an exact sequence of $R$-modules $$\mathbf{P_{-}}:\cdots\ra P^{-m}\xrightarrow{g^{-m}} \cdots\ra P^{-1}\xrightarrow{g^{-1}} \Omega^{n}G \ra 0,$$ where  $P^{j}$ is projective and $\ker(g^j)$ is a Gorenstein projective $R$-module in $\mathcal{A}$ for $j$ $\leq$ $-1$.

Gluing the two sequences $\mathbf{P}_{-}$ and $\mathbf{P}_{+}$ above, we obtain an exact sequence of projective $R$-modules $$\cdots \ra P^{-m}\xrightarrow{g^{-m}}\cdots \ra P^{-1}\xrightarrow{g^{-1}} \Omega^{n}P^{0}\xrightarrow{g^{0}}\cdots \ra \Omega^{n}P^{m}\xrightarrow{g^{m}} \cdots$$ such that  each  $\ker(g^j)$ is a Gorenstein projective $R$-module in $\mathcal{A}$.  Let $N$ $=$ $\oplus\ker(g^j)$.  Then $N$ is strongly Gorenstein projective by Lemma \ref{lem: 2017082101'}. It is clear that $N$ is in $\mathcal{A}$ and $\Omega^{n}G$ is a direct summand of $N$. This completes the proof.
\end{proof}

The following is a key result to prove Theorem \ref{thm: 181.100}.
\begin{prop} \label{lem: 1.9} Let $G$ be an $R$-module and $n$ a nonnegative integer. Then the following are equivalent:
\begin{enumerate}
\item There is a strongly Gorenstein projective $R$-module $N$ in $^\bot(G^{\bot_\infty})$ such that $\Omega^{n} G$ is a direct summand of $N$.
\item $^\bot(G^{\bot_\infty})$ $\subseteq $ ${\mathcal{GP}_n}$ and $\mathcal{K}_G$ $\subseteq $ ${\mathcal{P}_n}$.
\end{enumerate}
\end{prop}

\begin{proof} $(1)$ $\Rightarrow$ $(2)$. By (1), there is a strongly Gorenstein projective $R$-module $N$ in $^\bot(G^{\bot_\infty})$ such that $\Omega^{n} G$ is a direct summand of $N$. Then $\Omega^{n} G$ is Gorenstein projective. Thus $\Gpd_{R} G$ $\leq$ $n$, and so $^\bot(G^{\bot_\infty})$ $\subseteq $ ${\mathcal{GP}_n}$ by Lemma \ref{lem: 1.09}(3).

 For any $H$ $\in$ $\mathcal{K}_G$, to prove that $H\in{\mathcal{P}_n}$, by  Lemma \ref{lem: 1.6}(2), we only need to show that $\Omega^{n}H$ $\in$ $^\bot(N^\bot)$ $\cap$ $N^\bot$. It is clear that $H$ $\in$ $N^\bot$ since $N$ $\in$ $^\bot{(G^{\bot_{\infty}})}$ and  $H$ $\in$ $G^{\bot_{\infty}}$. Note that $N^\bot$ is thick by Lemma \ref{lem: 1.6}(1) and every projective $R$-module is in $N^\bot$. Thus $\Omega^{n}H$ $\in$ $N^\bot$. Next we will prove that $\Omega^{n}H$ is in $^\bot(N^\bot)$. Let $K$ be an $R$-module in $N^\bot$. Since $N$ is strongly Gorenstein projective, there is an exact sequence $0\ra N\ra P\ra N\ra 0$ with $P$ projective. It follows that $N^{\bot}=N^{\bot_\infty}$, and so $K$ $\in$ $N^{\bot_\infty}$.  Since $\Omega^{n} G$ is a direct summand of $N$, $\Ext^{i}_{R}(\Omega^{n}G,K)$ $=$ $0$ for $i$ $\geq$ $1$. It follows that $\Ext^{n+i}_{R}(G,K)$ $=$ $0$ for $i$ $\geq$ $1$. By Lemma \ref{lem: 1.09}(1), $\Ext^{n+1}_{R}(H,K)$ $=$ $0$. Thus $\Ext^{1}_{R}(\Omega^{n}H,K)$ $=$ $0$, and hence $\Omega^{n}H$ $\in$ $^\bot(N^\bot)$. So $\Omega^{n}H$ $\in$ $^\bot(N^\bot)$ $\cap$ $N^\bot$, as desired.

$(2)$ $\Rightarrow$ $(1)$. Note that $(^\bot{(G^{\bot_{\infty}})},G^{\bot_{\infty}})$ is a complete  hereditary cotorsion pair by Lemma \ref{lem: 20170819001}. The proof follows from Lemma \ref{lem: 1803261.1}.
\end{proof}

\noindent\begin{rem} \label{rem: 1.0009} \rm{
 Assume that $M$ is an $R$-module of finite projective dimension at most $n$ and  $N$ is a strongly Gorenstein projective $R$-module. Note that $\Omega^{n}(M\oplus N)$ can be taken to be $\Omega^{n}M \oplus N$. Since $\Omega^{n}M$ is projective, it is clear that $\Omega^{n}M \oplus N$  is a strongly Gorenstein projective $R$-module in $^\bot(({M\oplus N})^{\bot_\infty})$. Hence $M\oplus N$ satisfies the condition (1) of Proposition \ref{lem: 1.9}.}
\end{rem}

The following
well-known lemma will play a useful role in our investigations.

\begin{lem}{\rm (Dimension shifting)} \label{lem: 1.03} Let $R$ be a ring, n a positive integer, and $M$ an $R$-module. Fix an exact sequence of $R$-modules $0\ra K\ra T_{n}\ra\cdots\ra T_{1}\ra L\ra 0$. The following are true:
\begin{enumerate}

\item[(1)]  If $\Ext^i_R(M, T_s)$ $=$ $0$ for $1$ $\leq$ $s$ $\leq$ $n$ and $i$  $\geq$ $1$, then $\Ext_{R}^{n+j}(M,K)$ $\cong$ $\Ext_{R}^{j}(M,L)$ for all $j$ $\geq$ $1$.

\item[(2)] If $\Ext^i_R(T_s, M)$ $=$ $0$ for $1$ $\leq$ $s$ $\leq$ $n$ and $i$  $\geq$ $1$, then $\Ext_{R}^{n+j}(L,M)$ $\cong$ $\Ext_{R}^{j}(K,M)$ for all $j$ $\geq$ $1$.

\end{enumerate}
\end{lem}

Now we can give the proof of Theorem \ref{thm: 181.100}.
\begin{para}\label{3.15}{\bf{Proof of Theorem \ref{thm: 181.100}}}.
{$(1)$ $\Rightarrow$ $(2)$. It is only need to prove that $\mathcal{A}\subseteq {\mathcal{GP}_n}$. Let $M$ be an $R$-module in $\mathcal{A}$. Since $(\mathcal{A},\mathcal{B})$ is a complete cotorsion pair, there is an exact sequence of $R$-modules $$\xymatrix{0\ar[r]&M\ar[r]&X^{0}\ar[r]^{f^{0}}&\cdots\ar[r]&X^{m}\ar[r]^{f^{m}}&\cdots}$$ such that each $X^j$ $\in$ $\mathcal{K}_\mathfrak{C}$ and each $\ker(f^i)$ $\in$ $\mathcal{A}$. Note that there is an exact sequence $0\ra R\ra T_{0}\ra \cdots \ra T_{n}\ra 0 $ of $R$-modules with each $T_{i}$ $\in$ $\Add T$. Thus, for every free $R$-module $R^{(I)}$, we have the following exact sequence of $R$-modules
$$0\ra R^{(I)}\ra  T_{0}^{(I)}\ra \cdots \ra T_{n}^{(I)}\ra 0,$$
where each  $T_{i}^{(I)}$ $\in$ $\mathcal{B}$ since $\Add T$ $=$ $\mathcal{K}_\mathfrak{C}$ $\subseteq$ $\mathcal{B}$. It follows from Lemma \ref{lem: 1.03}(1) that $\Ext_{R}^{n+i}(A,R^{(I)})$ $\cong$ $\Ext_{R}^{i}(A, T_{n}^{(I)})$ $=$ $0$ for all $A$ $\in$ $\mathcal{A}$ and $i$ $\geq$ $1$. Hence $\Ext_{R}^{n+i}(A, P)$ $=$ $0$ for every projective $R$-module $P$ and $i$ $\geq$ $1$.  Therefore $\Ext_{R}^{n+i}(\ker(f^j), P)$ $=$ $0$ for every projective $R$-module $P$, $j$ $\geq$ $0$ and $i$ $\geq$ $1$. It is clear that $\pd_{R}X^{j}$ $\leq$ $n$ since $X^j$ $\in$ $\mathcal{K}_\mathfrak{C}$ $=$ $\Add T$. Thus $\Gpd_R M$ $\leq$ $n$ by Lemma \ref{lem: 1.8}.

$(2)$ $\Rightarrow$ $(3)$. Since $(\mathcal{A},\mathcal{B})$ is a complete cotorsion pair, there is an exact sequence $0\ra R \ra T_{0}\ra \cdots \ra T_{n-1}\ra A\ra 0$ of $R$-modules with $A$ $\in$ $\mathcal{A}$ and $T_{j}$ $\in$ $\mathcal{K}_\mathfrak{C}$ for $0\leq j\leq n-1$. Let $T$ $=$ $A$ $\oplus$ $\oplus^{n-1}_{0}T_{i}$. We will check that $T$ is an $n$-tilting $R$-module. Note that each $T_j$ $\in$ $\mathcal{B}$ and $\mathcal{A}\subseteq {\mathcal{GP}_n}$, $\Ext_{R}^{i}(M,A)$ $\cong$ $\Ext_{R}^{n+i}(M,R)$ $=$ $0$ for all $M$ $\in$ $\mathcal{A}$ and $i$ $\geq$ $1$ by Lemma \ref{lem: 1.03}(1).  It follows that $A$ $\in$ $\mathcal{B}$, and then $A$ $\in$ $\mathcal{K}_\mathfrak{C}$. Therefore $T$ $\in$ $\mathcal{K}_\mathfrak{C}$, and hence $T$ $\in$ $\mathcal{P}_{n}$ since $\mathcal{K}_\mathfrak{C}$ $\subseteq $ ${\mathcal{P}_n}$ by assumption. It is easy to check that $\Ext_{R}^{i}(T,T^{(\lambda)})$ $=$ $0$ for all $i$ $\geq$ $1$ and each cardinal $\lambda$ by noting that $\mathcal{K}_\mathfrak{C}$ is closed under direct sums. So $T$ is an $n$-tilting $R$-module and Add$T$ $\subseteq$ $\mathcal{K}_\mathfrak{C}$.

Applying Lemma \ref{lem: 1803261.1}, for each $L$ in $\mathcal{A}$, we obtain a strongly Gorenstein projective $R$-module ${N_L}$ in $\mathcal{A}$ such that $\Omega^nL$ is a direct summand of ${N_L}$. Let $\mathcal{X}$ be the class $\{$${N_L}$ $|$ $L$ $\in$ $\mathcal{A}$$\}$. Then $\{T\}$ $\cup$ $\mathcal{X}$ $\subseteq$ $\mathcal{A}$. Next we check that $\mathcal{B}$ $=$ $T^{\bot_\infty}$ $\cap$ $\mathcal{X}^{\bot_\infty}$. It is clear that $\mathcal{B}$ $\subseteq$ $ T^{\bot_\infty}$ $\cap$ $\mathcal{X}^{\bot_\infty}$. For the reverse containment, assume that $L$ is an $R$-module in $\mathcal{A}$ and $H$ is an $R$-module in $T^{\bot_\infty}$ $\cap$ ${{N_L}}^{\bot_\infty}$. By \cite[Theorem 3.11]{B}, there is an exact sequence of $R$-modules
 $$\xymatrix{0\ar[r] &K_{n-1}\ar[r]&T'_{n-1}\ar[r]&\cdots \ar[r]&T'_{1}\ar[r]&T'_{0}\ar[r]&H\ar[r]&0}$$ with $T'_{j}$ $\in$ Add$T$. It is clear that $\pd_R T'_{j}$ $\leq$ $n$ for $0$ $\leq$ $j$ $\leq$ $n-1$. Since ${N_L}$ is strongly Gorenstein projective, $T'_j$ $\in$ ${N_L}^{\bot_\infty}$  for  $0$ $\leq$ $j$ $\leq$ $n-1$.
 Note that $H$ $\in$ ${N_L}^{\bot_\infty}$ and ${N_L}^\bot$  $=$ ${N_L}^{\bot_\infty}$. Applying Lemma \ref{lem: 1.6}(1), $K_{n-1}$ $\in$ ${N_L}^{\bot_\infty}$. Since $\Omega^{n}L$ is a direct summand of ${N_L}$, $\Ext_{R}^{i}(\Omega^{n}L,K_{n-1})$ $=$ $0$, and so $\Ext_{R}^{n+i}(L,K_{n-1})$ $=$ $0$ for $i$ $\geq$ $1$.
Since $T'_j$ $\in$ $\mathcal{K}_\mathfrak{C}$ for $0\leq j\leq n-1$, $\Ext_{R}^{i}(L,H)$  $\cong$ $\Ext_{R}^{n+i}(L,K_{n-1})$ for $i$ $\geq$ $1$ by Lemma \ref{lem: 1.03}(1). Then $\Ext_{R}^{i}(L,H)=0$ for $i$ $\geq1$, and so $L$ $\in$ $^\bot(T^{\bot_{\infty}} \cap {N_L}^{\bot_{\infty}})$ $\subseteq$ $ ^\bot(T^{\bot_\infty} \cap \mathcal{X}^{\bot_\infty})$. It follows that $\mathcal{A}$ $\subseteq$ $ ^\bot(T^{\bot_\infty} \cap \mathcal{X}^{\bot_\infty})$, and then  $ T^{\bot_\infty}$ $\cap$ $\mathcal{X}^{\bot_\infty}$ $\subseteq$ $\mathcal{B}$. So  $\mathcal{B}$ $=$ $ T^{\bot_\infty}$ $\cap$ $\mathcal{X}^{\bot_\infty}$.

$(3)$ $\Rightarrow$ $(1)$. It is clear that $\Add T$  $\subseteq$ $\mathcal{A}$. Note that $T$ is a tilting $R$-module and each object in $\mathcal{X}$ is strongly Gorenstein projective. It follows that $\Add T$  $\subseteq$ $\mathcal{X}^{\bot_\infty}$, and so $\Add T$ $\subseteq$ $\mathcal{K}_\mathfrak{C}$ since  $\mathcal{B}$ $=$ $ T^{\bot_\infty}$ $\cap$ $\mathcal{X}^{\bot_\infty}$. For the reverse containment, we assume that $K$ $\in$ $\mathcal{K}_\mathfrak{C}$.
 Let $M$ be an $R$-module in $T^{\bot_{\infty}}$. Then there exists an exact sequence $0\ra L\ra T_{n-1} \ra \cdots \ra T_{0}\ra M\ra 0$ of $R$-modules with
$T_{j}$ $\in$ $\Add T$ for $0\leq j\leq n-1$ by \cite[Theorem 3.11]{B}. Since $\Add T$ $\subseteq$ $\mathcal{K}_\mathfrak{C}$ by the proof above, $T_{j}$ $\in$ $\mathcal{K}_\mathfrak{C}$ for $0\leq j\leq n-1$. Thus  each $T_{j}$ is in $K^{\bot_{\infty}}$, and so $\Ext_{R}^{i}(K,M)$ $\cong$ $\Ext_{R}^{n+i}(K,L)$ for $i$ $\geq$ $1$ by Lemma \ref{lem: 1.03}(1).  By assumption, $\mathcal{K}_\mathfrak{C}$ $\subseteq$ $\mathcal{P}_n$. Then $K$ $\in$  $\mathcal{P}_n$. Thus  $\Ext_{R}^{i}(K,M)$ $\cong$ $\Ext_{R}^{n+i}(K,L)$ $=$ $0$  for $i$ $\geq$ $1$. Therefore $K$ $\in$ $^\bot(T^{\bot_{\infty}})$. It is clear that $K$ $\in$ $T^{\bot_{\infty}}$. Then $K$ $\in$ $\Add T$ by noting that $\mathcal{K}_T$ $=$ $\Add T$. So $\mathcal{K}_\mathfrak{C}$ $\subseteq$ $\Add T$, as desired.

$(3)$ $\Rightarrow$ $(4)$. By assumption, $\mathcal{B}$ $=$ $G^{\bot_\infty}$ for an $R$-module $G$. Using similar arguments as in the proof of  $(2)$ $\Rightarrow$ $(3)$, one can check that $\mathcal{B}$ $=$ $T^{\bot_{\infty}}\cap {N_G}^{\bot_{\infty}}$, where $N_G$ is a strongly Gorenstein $R$-module such that $\Omega^nG$ is a direct summand of $N_G$.

$(4)$ $\Rightarrow$ $(3)$. Note that $T^{\bot_{\infty}}\cap N^{\bot_{\infty}}$ $=$ $(T\oplus N)^{\bot_{\infty}}$.  By Proposition \ref{lem: 1.9} and Remark \ref{rem: 1.0009}, we obtain the desired result.

$(2)$ $\Rightarrow$ $(5)$ follows from Lemma \ref{lem: 1803261.1}.

$(5)$ $\Rightarrow$ $(2)$ holds by Proposition \ref{lem: 1.9} by noting  that $\mathcal{K}_G$ $=$ $\mathcal{K}_\mathfrak{C}$ and $^\bot(G^{\bot_\infty})$ $=$ $\mathcal{A}$.

$(5)$ $\Rightarrow$ $(6)$. By assumption, we only need to prove that $(\Omega^n G)^{\bot_\infty}$ $=$ $M^{\bot_\infty}$. It is clear that  $M^{\bot_\infty}$  $\subseteq$ $(\Omega^n G)^{\bot_\infty}$ since $\Omega^n G$ is a direct summand of $M$. Let $L$ be an $R$-module in $(\Omega^n G)^{\bot_\infty}$. Then $\Ext_{R}^{n+i}(G,L)$ $=$ $0$ for $i$ $\geq$ $1$. By Lemma \ref{lem: 1.09}, $\Ext_{R}^{n+i}(M,L)$ $=$ $0$ for $i$ $\geq$ $1$. Since $M$ is strongly Gorenstein projective, there is an exact sequence $0\ra M\ra P\ra M\ra 0$ with $P$ projective. It follows that $\Ext_{R}^{j}(M,L)$ $=$ $0$ for $j$ $\geq$ $1$. Then we have $(\Omega^n G)^{\bot_\infty}$ $=$ $M^{\bot_\infty}$.

$(6)$ $\Rightarrow$ $(5)$. By assumption, $N$ is strongly Gorenstein projective and $(\Omega^n G)^{\bot_\infty}$ $=$ $N^{\bot_\infty}$. Note that $N^\bot$ $=$ $N^{\bot_\infty}$. It follows that $\Omega^n G$ $\in$ $^\bot(N^{\bot})$. Applying Lemma \ref{lem: 1.6}(2), $(^\bot(N^{\bot}),N^{\bot})$ is a projective cotorsion pair. By Lemma \ref{lem: 1803261.1}, there is a strongly Gorenstein projective $R$-module $M$ in $^\bot(N^{\bot})$ such that $\Omega^0(\Omega^n G)$ is a direct summand of $M$. It is clear that $\Omega^0(\Omega^n G)$ $=$ $\Omega^n G$, and so $\Omega^n G$ is a direct summand of $M$. By assumption, $N$ $\in$ $^\bot(G^{\bot_\infty})$. It follows  that $M$ $\in$ $^\bot(N^{\bot})$ $=$ $^\bot(N^{\bot_\infty})\subseteq {^\bot(G^{\bot_\infty})}$. Hence $M$ is in $^\bot(G^{\bot_\infty})$. This completes the proof.
}\hfill$\Box$

\end{para}

\begin{cor}\label{cor:1.111112} Let $T$ be an $n$-tilting $R$-module and $N$ a strongly Gorenstein projective $R$-module. Then $\mathcal{K}_{T\oplus N}$ $=$ $\Add T$.
\end{cor}
\begin{proof} This can be checked by the proof of $(3)$ $\Rightarrow$ $(1)$ in Theorem \ref{thm: 181.100}.
\end{proof}

  Given an infinite cardinal number $\lambda$, an $R$-module $M$ is said to be \emph{$\lambda^{<}$-generated} if it is generated by less than $\lambda$ elements. Furthermore, we have the following corollary.

\begin{cor}\label{cor:180401.01}  Let $\mathfrak{C}$ $=$ $(\mathcal{A},\mathcal{B})$ be a complete hereditary cotorsion pair of $R$-modules such that $\mathcal{K}_\mathfrak{C}$ $=$ $\Add T$ for a tilting $R$-module $T$. If each Gorenstein projective $R$-module is a direct sum of \emph{$\lambda^{<}$-generated} $R$-modules for some infinite cardinal $\lambda$, then there is an $R$-module $G$ such that $\mathcal{B}$ $=$ $G^{\bot_\infty}$.
\end{cor}
\begin{proof} Note that $\mathcal{K}_\mathfrak{C}$ $=$ $\Add T$ for a tilting $R$-module by hypothesis.  Using Theorem \ref{thm: 181.100},  $\mathcal{B}$ $=$ $ T^{\bot_\infty}$ $\cap$ $\mathcal{X}^{\bot_\infty}$, where $\mathcal{X}$ is a class of strongly Gorenstein projective $R$-modules. By assumption, each Gorenstein projective $R$-module is a direct sum of $\lambda^{<}$-generated modules for an infinite cardinal number $\lambda$. Thus each module in $\mathcal{X}$ is a direct sum of  such modules. It is easy to see that there is an $R$-module $M$ such that $\mathcal{X}^{\bot_\infty}$ $=$ $M^{\bot_\infty}$. Let $G$ $=$ $T\oplus M$. It is clear that $\mathcal{B}$ $=$ $G^{\bot_\infty}$.
\end{proof}

\section{\bf Applications}

Following \cite{Bieri}, we denote by $\mathcal{FP}_{\infty}$ the class of $R$-modules possessing a projective resolution consisting of
 finitely generated modules. The objects of $\mathcal{FP}_{\infty}$ are sometimes referred to as the \emph{finitely ${\infty}$-presented} modules
 (see \cite{BP}). If $R$ is left Noetherian $($left coherent$)$, then  $\mathcal{FP}_{\infty}$ coincides with the class of finitely generated $($finitely presented$)$ $R$-modules.
\begin{lem} \label{lem: 2017082101}Let $R$ be a ring and $M$ be a Gorenstein projective $R$-module. If $M$ $\in$ $\mathcal{FP}_{\infty}$, then there is a strongly Gorenstein projective $R$-module $N$ such that $M$ is a direct summand of $N$, where $N$ is a direct sum of finitely generated Gorenstein projective $R$-modules.

\end{lem}
\begin{proof} Let $M$ be a Gorenstein projective $R$-module in $\mathcal{FP}_{\infty}$.
By
\cite[Lemma 2.7]{JwandliHuNonG}, there is an exact sequence $\mathbf{P_{+}}:$ $0\ra M \ra P^{0}\xrightarrow{f^0} \cdots \ra P^{n}\xrightarrow{f^n} \cdots$ with $P^i$ finitely generated projective and  $\ker(f^i)$ finitely generated Gorenstein projective for $i$ $\geq$ $0$. Note that $M$ is in $\mathcal{FP}_{\infty}$. It follows from \cite[Theorem 2.5]{H} that there exists an exact sequence $\mathbf{P_{-}}:$ $\cdots \ra P^{-n}\xrightarrow{f^{-n}} \cdots \ra P^{-1}\xrightarrow{f^{-1}} M\ra 0$, where  $P^i$ is finitely generated projective and  $\ker(f^i)$ is finitely generated Gorenstein projective for $i$ $\leq$ $-1$. Gluing the two sequences $\mathbf{P}_{-}$ and $\mathbf{P}_{+}$ above, we obtain an exact sequence $ \cdots\ra P^{-1}\xrightarrow{f^{-1}} P^{0}\xrightarrow{f^{0}} P^1\xrightarrow{f^{1}} \cdots$ of finitely generated projective $R$-modules with each $\ker(f^i)$ finitely generated Gorenstein projective. Let $N$ $=$ $\oplus \ker(f^i)$. It follows that $N$ is strongly Gorenstein projective by Lemma \ref{lem: 2017082101'} and $M$ is a direct summand of $N$.
\end{proof}

The following lemma is essentially taken from \cite[Lemma 4.4]{XiCC}, where a variation
of it appears. The proof given there carries over to the present situation.

\begin{lem}\label {lem:17062201} Let $R$ be a left Noetherian ring. Then {\rm{findim}}$R$ $=$ {\rm sup}$\{\Gpd_R M \ | \ M $ is finitely generated and $ \Gpd_R M < \infty\}$.
\end{lem}

We are now in a position to prove Theorem \ref{thm:1706220201}.
\begin{para}\label{5.2}{\bf{Proof of Theorem \ref{thm:1706220201}}}.
 $(1)\Rightarrow(2)$ follows from \cite[Theorem 2.6]{Anegeleri} and Lemma \ref{lem: 2.5''}.

  $(2)\Rightarrow(1)$ holds by Theorem \ref{thm: 181.100} by noting that $\Gpd_R M$ $=$ $\pd_R M$ whenever $\pd_R M<\infty$.

$(1)\Rightarrow(3)$. It is clear that there is a set $\mathcal{S}$ $\subseteq$ ${\mathcal{GP}^{<\infty}}$ such that each module in ${\mathcal{GP}^{<\infty}}$ is isomorphic to an element in $\mathcal{S}$. Let $G$ $=$ $\oplus G_i$ be the direct sum of all elements $G_i$ in $\mathcal{S}$. Then $({\mathcal{GP}^{<\infty}})^{\bot_\infty}$  $=$ $\mathcal{S}^{\bot_\infty}$ $=$ $G^{\bot_\infty}$, and so $\mathfrak{C}=(^\bot((\mathcal{GP}^{<\infty})^{\bot_\infty}),\mathcal{GP}^{\bot_\infty})$ is a  complete hereditary cotorsion pair.

Assume that findim$R$ $=$ $n$ $<$ $\infty$. Then each $G_i$  $\in$ $\mathcal{GP}_n$ by Lemma \ref{lem:17062201}. It is clear that $\Omega^nG_i$ can be taken to be  finitely generated  and Gorenstein projective. For each $\Omega^nG_i$, by Lemma \ref{lem: 2017082101}, there is a strongly Gorenstein projective $R$-module $N_i$  such that $N_i$ is a direct sum of finitely generated Gorenstein projective $R$-modules and $\Omega^nG_i$ is a direct summand of $N_i$. Let $\Omega^nG$ $=$ $\oplus \Omega^n G_i$  and $N$ $=$ $\oplus N_i$. It is clear that $N$ is strongly Gorenstein projective and  $\Omega^nG$ is a direct summand of $N$. By the construction, $N$ is also a direct sum of finitely generated Gorenstein projective $R$-modules. Note that each finitely generated Gorenstein projective $R$-module is in ${\mathcal{GP}^{<\infty}}$. It follows that $N$ $\in$ $^\bot{({(\mathcal{GP}^{<\infty})}^{\bot_\infty})}$ since ${\mathcal{GP}^{<\infty}}$ $\subseteq$ $^\bot{({(\mathcal{GP}^{<\infty})}^{\bot_\infty})}$. Since $R$ is left Noetherian and each $G_i$ is finitely generated, it is easy to check that $\mathcal{K}_{\mathfrak{C}}$ $=$ $\mathcal{K}_G$ is closed under direct sums. By Theorem \ref{thm: 181.100}, $\mathcal{K}_{\mathfrak{C}}$ $=$ $\Add T$, where $T$ is a tilting $R$-module.

$(3)\Rightarrow(1)$. By hypothesis and Theorem \ref{thm: 181.100}, $^\bot(({\mathcal{GP}^{<\infty})}^{\bot_\infty})$ $\subseteq$ $\mathcal{GP}_n$. It follows that ${\mathcal{GP}^{<\infty}}$ $\subseteq$ $\mathcal{GP}_n$. So \rm{findim}$R$ $\leq$ $n$  by Lemma \ref{lem:17062201}.\hfill$\Box$
\end{para}

Recall that an $R$-module $\omega$ is said to be a \emph{Wakamatsu tilting module} \cite{Mantese,Wamastilting} if it has the following properties:

\begin{enumerate}
\item[$(W1)$] there exists an exact sequence
$\cdots\ra P_i\ra\cdots\ra  P_0\ra \omega\ra 0$
with  $P_i$ finitely generated and projective for  $i$ $\geq$ $0$;
\item [$(W2)$] $\Ext_{R}^{i}(\omega,\omega)$ $=$ $0$ for  $i$ $\geq$ $1$;

\item[$(W3)$] there exists an exact sequence $0\ra R\ra \omega_0 \xrightarrow{f_0}  \cdots \ra \omega_i\xrightarrow{f_i}\cdots$ with $\omega_i$ $\in$ add$\omega$
and $\ker(f_i)$ $\in$ $^{\bot_\infty}\omega$ for $i$ $\geq$ $0$.
\end{enumerate}


The following result contains  Theorem \ref{prop2017081202'}  from the introduction.

\begin{thm} \label{prop2017081202} Let $R$ be a ring and $\omega$ a Wakamatsu tilting $R$-module. Fix an exact sequence $0\ra R\ra \omega_0 \xrightarrow{f_0} \cdots \ra \omega_i\xrightarrow{f_i}\cdots$ with $\omega_i$ $\in$ {\rm add$\omega$}
and $\ker (f_i)$ $\in$ $^{\bot_\infty}\omega$ for  $i$ $\geq$ $0$. If we set $A$ $=$ ${\bigoplus\limits_{i\geqslant0}}\ker (f_i)$, then the following are equivalent for any nonnegative integer $n$:
\begin{enumerate}
\item $\omega$ is an $n$-tilting $R$-module.
\item $\omega^{\bot_> n}$ $=$ $A^{\bot_> n}$ and $\mathcal{K}_A$ $=$ $\Add T$, where $T$ is an $n$-tilting $R$-module.
\item $\omega^{\bot_> n}$ $\supseteq$ $A^{\bot_> n}$ and $\mathcal{K}_A$ $=$ $\Add T$, where $T$ is an $n$-tilting $R$-module.
\item $\mathcal{K}_{\omega\oplus A}$ $=$ $\Add T$, where $T$ is an $n$-tilting $R$-module.

\end{enumerate}
\end{thm}

\begin{proof}$(1)$ $\Rightarrow$ $(2)$. Assume that $\omega$ is an $n$-tilting $R$-module. To prove $\omega^{\bot_>n}$ $=$ $A^{\bot_> n}$, it is sufficient to show that $A$ $\in$ $\mathcal{P}_n$.  Note that $\ker (f_i)$ $\in$ $\mathcal{FP}_{\infty}$ for any $i\geq0$ by \cite[Theorem 1.8]{BP}. It follows from \cite[Lemma 3.1.6]{GT} that $A^{\bot_\infty}$ is closed under direct sums. Hence $\mathcal{K}_A$ is closed under direct sums. Let $P$ be any projective $R$-module. Since $\omega$ is an $n$-tilting $R$-module, there is an exact sequence $$0\ra P\ra T_{0}\ra \cdots\ra T_{n}\ra 0$$ with $T_{i}$ $\in$ $\Add \omega$ for $0\leq i\leq n$. By the proof above, $A ^{\bot_\infty}$ is closed under direct sums. Then $\Add \omega$ $\subseteq$ $A^{\bot_\infty}$. It follows from Lemma \ref{lem: 1.03}(1) that $\Ext_{R}^{n+j}(A,P)$ $\cong$ $\Ext_{R}^{j}(A,T_n)$ for $j\geq1$, and so $\Ext_{R}^{n+j}(A,P)$ $=$ $0$ for $j\geq1$. Hence $\Ext_{R}^{n+j}(\ker(f_i),P)$ $=$ $0$ for $i\geq0$ and $j\geq1$. On the other hand, since $\textrm{pd}_{R}\omega_i\leq n$ for any $\omega_i$ $\in$  add$\omega$, we have that $\textrm{pd}_{R}\ker(f_i)<\infty$. This implies that each  $\textrm{pd}_{R}\ker(f_i)\leq n$  by \cite[Theorem 2.20 and Proposition 2.27]{H}, and then $A$ $\in$ $\mathcal{P}_n$. By Lemma \ref{lem: 1.09}, $^\bot(A^{\bot_\infty})$ $\subseteq$ $\mathcal{P}_n$. Applying Theorem \ref{thm: 181.100}, we obtain that  $\mathcal{K}_A$ $=$ $\Add T$.

$(2)$ $\Rightarrow$ $(3)$ is clear.

 $(3)$ $\Rightarrow$ $(4)$. By Theorem \ref{thm: 181.100}, there is a strongly Gorenstein projective $R$-module $N$ in $^\bot(A^{\bot_\infty})$ such that  $(\Omega^n A)^{\bot_\infty}$ $=$ $N^{\bot_\infty}$. By assumption, $\omega^{\bot_>n}$ $\supseteq$ $A^{\bot_> n}$. It follows that $(\Omega^n \omega)^{\bot_\infty}$ $\supseteq$ $(\Omega^n A)^{\bot_\infty}$. Then $(\Omega^n(\omega\oplus A))^{\bot_\infty}$ $=$ $(\Omega^n\omega)^{\bot_\infty}$ $\cap$ $(\Omega^n A)^{\bot_\infty}$ $=$ $N^{\bot_\infty}$. It is clear that $N$ is in $^\bot((\omega \oplus A)^{\bot_\infty})$. Using similar arguments as in the proof of  $(1)$ $\Rightarrow$ $(2)$, one can check that $\mathcal{K}_{\omega\oplus A}$ is closed under direct sums. By Theorem \ref{thm: 181.100}, we obtain that $\mathcal{K}_{\omega\oplus A}$ $=$ $\Add T$, where $T$ is an $n$-tilting $R$-module.

$(4)$ $\Rightarrow$ $(1)$. Since $\Ext_{R}^{i}(\im (f_{n}),N)=0$ for any $N$ $\in$ add$\omega$ and $i\geq1$ by hypothesis, $\Ext^1_R(\im (f_n),\im(f_{n-1}))$ $\cong$ $\Ext^{n+1}_R(\im(f_n),R)$ by Lemma \ref{lem: 1.03}(1). It follows from Theorem \ref{thm: 181.100} that $\im (f_{n})\in{\mathcal{GP}_n}$, and hence $\Ext^1_R(\im (f_n),\im (f_{n-1}))=0$. Consequently, we obtain the following exact sequence $$0\ra R\ra \omega_0 \xrightarrow{f_0} \omega_1 \xrightarrow{f_1} \cdots \ra \omega_{n-1}\xrightarrow{f_{n-1}}\im (f_{n-1})\ra 0$$ with $\im (f_{n-1})$ $\in$ add$\omega$ and $\omega_i$ $\in$ add$\omega$ for $0\leq i\leq n-1$. So $\omega$ is an $n$-tilting $R$-module.
\end{proof}

\begin{cor}\label{cor:2018081601} Let $\omega$ be a Wakamatsu tilting  $R$-module with finite projective dimension.  Keep the notations as in Theorem \ref{prop2017081202}. Then the following are equivalent:
\begin{enumerate}
\item $\omega$ is a  tilting  $R$-module.
\item  $\mathcal{K}_A $ $=$ $\Add T$, where $T$ is a tilting $R$-module.
\item  $\mathcal{K}_{\omega\oplus A}$ $=$ $\Add T$, where $T$ is a tilting $R$-module.
\end{enumerate}
\end{cor}

\noindent\begin{rem} \label{rem: 2017081701} \rm{
It is still an open problem whether a Wakamatsu tilting $R$-module of
finite projective dimension must be a tilting $R$-module whenever $R$ is an Artin algebra. This is known as Wakamatsu Tilting
Conjecture (see \cite[Chapter IV]{Beli and Reiten}). Mantese and Reiten  \cite{Mantese} showed that the Wakamatsu Tilting Conjecture is a special case of the Finitistic Dimension Conjecture. These conjectures are also related to many other homological conjectures and attract
many algebraists, see for instance \cite{Anegeleri,Beli and Reiten,Celikbas,Mantese,Wei2011,XiCC,Zimmermann}.}
\end{rem}
 Let $R$ be a ring. A left $R$-module $M$ is called  \emph{$\mathcal{FP}_{\infty}$-injective} (or \emph{absolutely clean}) \cite{BGGIAC} if $\Ext^1_R(N,M)$ $=$ $0$ for all $R$-modules $N$ $\in$ $\mathcal{FP}_{\infty}$. Let $\mathcal{FP}_{\infty}$-Inj be the class of $\mathcal{FP}_{\infty}$-injective $R$-modules, then $\mathcal{FP}_{\infty}$-Inj $=$ $(\mathcal{FP}_{\infty})^{\bot_{\infty}}$ by \cite[Propositon 2.7]{BGGIAC}. It is clear that there is an $R$-module $G$ such that $\mathcal{FP}_{\infty}$-Inj $=$ $G^{\bot_{\infty}}$.  Thus $(^\bot(\mathcal{FP}_{\infty}$-Inj), $\mathcal{FP}_{\infty}$-Inj$)$ is a complete hereditary cotorsion pair over a general ring.

  Recall that a left $R$-module $M$ is \emph{$FP$-injective} (or \emph{absolutely pure}) \cite{MA,S} provided that ${\rm Ext}_R^{1}(N,M)=0$ for all finitely presented left $R$-modules $N$. If $R$ is left coherent, then  $\mathcal{FP}_{\infty}$-injective $R$-modules coincides with the class of $FP$-injective $R$-modules.

The \emph{$FP$-injective dimension} of $M$ is defined to be the least nonnegative integer $n$ such that $\Ext^{n+1}_{R}(N,M)$ $=$ $0$ for all finitely presented left $R$-modules $N$.

A coherent ring $R$ is called \emph{$n$-$FC$} \cite{DC cohernt} if $R$ has left and right $FP$-injective dimension at most $n$.  In the case of Noetherian rings, an $n$-$FC$ ring coincides with an $n$-Gorenstein ring originally defined by Iwanaga in \cite{Iwer}.

\begin{prop} \label{prop:1832801} Let $R$ be a ring and $\mathcal{K}_\mathfrak{C}$ the kernel of $\mathfrak{C}$ $=$ $(^\bot(\mathcal{FP}_{\infty}$-Inj), $\mathcal{FP}_{\infty}$-Inj$)$. Then the following are true for any nonnegative integer $n$:
\begin{enumerate}
\item $\mathcal{FP}_{\infty}$ $\subseteq$ $\mathcal{GP}_n$ if and only if $\mathcal{K}_\mathfrak{C}$ $=$ $\Add T$, where $T$ is an $n$-tilting $R$-module.
\item If $R$ is a commutative coherent ring, then $R$ is $n$-FC if and only if $\mathcal{K}_\mathfrak{C}$ $=$ $\Add T$, where $T$ is an $n$-tilting $R$-module.
\end{enumerate}
\end{prop}

\begin{proof} $(1)$ $``\Rightarrow"$. Assume that $\mathcal{FP}_{\infty}$ $\subseteq$ $\mathcal{GP}_n$. It is clear that $\mathcal{FP}_{\infty}$-Inj $=$ $G^{\bot_{\infty}}$ for an $R$-module $G$. Using similar arguments as in the proof of  Theorem \ref{thm:1706220201}, there is a strongly Gorenstein projective $R$-module $N$ in  $^\bot(\mathcal{FP}_{\infty}$-Inj) such that $\Omega^n G$ is a direct summand of $N$. It is clear that $\mathcal{K}_\mathfrak{C}$ is closed under direct sums.  By Theorem \ref{thm: 181.100}, $\mathcal{K}_\mathfrak{C}$ $=$ $\Add T$, where $T$ is an $n$-tilting $R$-module.

$``\Leftarrow"$. The proof follows from Theorem \ref{thm: 181.100}.

$(2)$ Let $\mathcal{FP}$ be the class of finitely presented $R$-modules. If $R$ is a commutative coherent ring, then $\mathcal{FP}$ $=$ $\mathcal{FP}_{\infty}$ and  $\mathcal{FP}_{\infty}$-injective $R$-modules coincides with the class of $FP$-injective $R$-modules. The proof follows from \cite[Theorem 7]{DC cohernt} and $(1)$.
\end{proof}

Recall that an $R$-module $M$ is called \emph{Gorenstein injective} \cite{EJ GP} if there exists an exact sequence $\mathbf{I}:\cdots\ra I_1\ra I_0\ra I^0\ra I^1\ra \cdots$ of injective $R$-modules with $M\cong \textrm{im}(I_0\ra I^0)$ such that $\Hom_{R}(E,\mathbf{I})$ is exact for every injective $R$-module $E$.

 Let $\mathcal{GI}$ be the class of Gorenstein injective $R$-modules. By \cite[Theorem 4.6]{JJarxiv}, the cotorsion pair $\mathfrak{C}$ $=$ ($^\bot\mathcal{GI}$, $\mathcal{GI}$) is complete. It is easy to check that $\mathfrak{C}$ is hereditary and $\mathcal{K}_\mathfrak{C}$ coincides with the class of injective $R$-modules. The following result contains $(1)\Leftrightarrow(2)\Leftrightarrow(3)$ of Theorem \ref{thm:1.4} from the introduction.

\begin{prop} \label{prop20170806001} Let $R$ be a ring and $\mathcal{I}$ the class of injective $R$-modules. The following are equivalent for a nonnegative integer $n$:
\begin{enumerate}
\item $R$ is left Noetherian and $R$-{\rm Mod} $=$ $\mathcal{GP}_n$.
\item $\mathcal{K}_\mathfrak{C}$ $=$ $\Add T$, where $\mathfrak{C}$ $=$ $(^\bot\mathcal{GI}$, $\mathcal{GI})$ and $T$ is an $n$-tilting $R$-module.
\item  $\mathcal{K}_\mathfrak{C}$ $=$ $\Add T$, where $\mathfrak{C}$ $=$ {\rm($R$-{\rm Mod}, $\mathcal{I}$)} and $T$ is an $n$-tilting $R$-module.
\end{enumerate}
\end{prop}

\begin{proof} $(1)$ $\Rightarrow$ $(2)$. It is easy to check that  $\mathcal{K}_\mathfrak{C}$ $=$ $\mathcal{I}$. Then $\mathcal{K}_\mathfrak{C}$ is closed under direct sums since $R$ is left Noetherian. By  \cite[Theorem 4.1]{I on the fin}, $\mathcal{K}_\mathfrak{C}$ $\subseteq$ $\mathcal{P}_n$. Applying Theorem \ref{thm: 181.100}, we obtain that $\mathcal{K}_\mathfrak{C}$ $=$ $\Add T$, where $T$ is an $n$-tilting $R$-module.

$(2)$ $\Rightarrow$ $(3)$ holds by  noting that the kernel of the cotorsion pair $(^\bot\mathcal{GI}$, $\mathcal{GI})$ is the class of injective $R$-modules.

$(3)$ $\Rightarrow$ $(1)$. By assumption, $\mathcal{I}$ is closed under direct sums. It follows that  $R$ is left Noetherian. By Theorem \ref{thm: 181.100}, we have that $R$-Mod $=$ $\mathcal{GP}_n$. This completes the proof.
\end{proof}

Let $R$ be a commutative Noetherian ring of finite Krull dimension, in \cite{EHuATotal}, it is proved that $R$ is a Gorenstein ring if and only if every acyclic complex of projective $R$-modules is totally acyclic. We end this paper with the following result which contains $(1)\Leftrightarrow(4)\Leftrightarrow(5)$ of Theorem \ref{thm:1.4} from the introduction.

\begin{prop} \label{prop20170509001} Consider the following
conditions for a ring $R$:
\begin{enumerate}
\item $R$ is a Gorenstein ring.

\item For any exact sequence  $ \cdots\ra T_2 \xrightarrow{d_2} T_1 \xrightarrow{d_1}T_0 \xrightarrow{d_0} \cdots$ of tilting $R$-modules, $\mathcal{K}_{\mathfrak{C}}$ $=$ $\Add T$, where $\mathfrak{C}$ $=$ $(^\bot((\oplus\ker(d_i))^{\bot_\infty}),(\oplus\ker(d_i))^{\bot_\infty})$  and $T$ is a tilting $R$-module.
    \item For any exact sequence  $ \cdots\ra T_2 \xrightarrow{d_2} T_1 \xrightarrow{d_1}T_0 \xrightarrow{d_0} \cdots$ of tilting $R$-modules, each $\ker(d_i)$ has finite Gorenstein projective dimension.
\end{enumerate}
Then $\textit{(1)}\Rightarrow \textit{(2)}\Rightarrow \textit{(3)}$. The converses hold if $R$
is a commutative Noetherian ring of finite Krull dimension.
\end{prop}
\begin{proof}
$(1)$ $\Rightarrow$ $(2)$.  Let  $\cdots\ra T_2 \xrightarrow{d_2} T_1 \xrightarrow{d_1}T_0 \xrightarrow{d_0} \cdots$ be an exact sequence of tilting $R$-modules and $G=\oplus\ker(d_i)$. Since $R$ is Gorenstein by hypothesis, there exists a nonnegative integer $n$ such that each $R$-module has finite Gorenstein projective dimension at most $n$ by \cite[Theorem 12.3.1]{EJ}. So $\ker(d_i)$ $\in$ $\mathcal{GP}_n$  and $T_i$ $\in$ $\mathcal{GP}_n$ for $i\in{\mathbb{Z}}$.
Thus $T_i$ $\in$ $\mathcal{P}_n$ for $i\in{\mathbb{Z}}$ by \cite[Proposition 2.27]{H}.
By the proof of $(2)$ $\Rightarrow$ $(1)$ in Lemma \ref{lem: 1.8},  there exists an exact sequence $\cdots\ra P_2\xrightarrow{\beta_2} P_1\xrightarrow{\beta_1}P_0\xrightarrow{\beta_0} \cdots$ of projective $R$-modules with each $\ker(\beta_i)$ Gorenstein projective such that $\Omega^n\ker(d_i)$ $\cong$ $\ker (\beta_i)$ for all $i\in{\mathbb{Z}}$ and $\Omega^nG$ $\cong$ $\oplus\ker (\beta_i)$. Hence $\Omega^nG$  is strongly Gorenstein projective by Lemma \ref{lem: 2017082101'}. It is clear that $\Omega^nG$ $\in$ $^\bot(G^{\bot_\infty})$ since $(^\bot(G^{\bot_\infty}), G^{\bot_\infty})$ is hereditary. Thus, to show that $\mathcal{K}_{\mathfrak{C}}$ $=$ $\Add T$, it suffices to show that $\mathcal{K}_G$ is closed under direct sums by Theorem \ref{thm: 181.100}.

 Let $H$ $=$ $\oplus H_j$ with $H_j$ $\in$ $\mathcal{K}_G$.  It is clear that $H$ $\in$  $^\bot(G^{\bot_\infty})$. To prove that $H$ $\in$ $G^{\bot_\infty}$, we only need to show that $H$ $\in$ $\ker(d_i)^{\bot_\infty}$ for all $i$ $\in$ $\mathbb{Z}$. Note that each $H_j$ $\in$ $\ker(d_i)^{\bot_\infty}$ for all $i$ $\in$ $\mathbb{Z}$. Applying $\Hom_R(-,H_j)$ to the short exact sequence $0\ra \ker(d_{i})\ra T_{i}\ra \ker(d_{i-1})\ra 0$, one can check that $H_j$ $\in$ $T_i^{\bot_\infty}$ for each $i$ $\in$ $\mathbb{Z}$. Thus $H$ $\in$ $T_i^{\bot_\infty}$ for  $i$ $\in$ $\mathbb{Z}$ since each $T_i$ is $n$-tilting.  Consider the exact sequence $0\ra \ker(d_i)\ra T_i\ra T_{i-1}\ra \cdots\ra T_{i-n+1}\ra \ker(d_{i-n})\ra 0$, we have $\Ext^{k}_R(\ker(d_i),H)$ $\cong$ $\Ext^{n+k}_R(\ker(d_{i-n}),H)$ for $k$ $\geq$ $1$ by Lemma \ref{lem: 1.03}(2). Note that $\mathcal{K}_G$ $\subseteq$ $\mathcal{P}_n$ by Proposition \ref{lem: 1.9}. It follows that $H_j$ $\in$ $\mathcal{P}_n$, and hence $H$ $\in$ $\mathcal{P}_n$. Since $\Gpd_R\ker(d_{i-n})$ $\leq$ $n$ for $i$ $\in$ $\mathbb{Z}$ by the proof above, $\Ext^{n+k}_R(\ker(d_{i-n}),H)=0$ for $k$ $\geq$ $1$ by \cite[Theorem 2.20]{H}. Hence $\Ext^{k}_R(\ker(d_i),H)=0$ for $k$ $\geq$ $1$ and so $H$ $\in$ $\ker(d_i)^{\bot_\infty}$ for all $i$ $\in$ $\mathbb{Z}$, as desired.

$(2)$ $\Rightarrow$ $(3)$ follows from Theorem \ref{thm: 181.100}.

$(3)$ $\Rightarrow$ $(1)$.  Assume that $R$
is a commutative Noetherian ring of finite Krull dimension. Thus, to prove that $R$ is Gorenstein, it suffices  to show that every acyclic complex of projective $R$-modules is totally acyclic by  \cite[Corollary 2]{EHuATotal}.
 Let $\mathbf{P}$ be an acyclic complex of projective $R$-modules, it must be a direct summand of an acyclic complex $\mathbf{F}$ of free $R$-modules. Let $N$ be the direct sum of all cycles of the complex $\mathbf{F}$. Using similar arguments as in the proof of  \cite[Theorem 2.7]{DB}, $N$ is also a cycle of another acyclic complex $\mathbf{F'}$ of free $R$-modules. By assumption, there is a nonnegative integer $n$ such that $N$ $\in$ $\mathcal{GP}_n$. Thus each cycle of $\mathbf{F}$ has Gorenstein projective dimension at most $n$. It implies that each cycle of the complex $\mathbf{F}$ must be Gorenstein projective and so the complex $\mathbf{F}$ is totally acyclic. It follows that $\mathbf{P}$ is also totally acyclic.  This completes the proof.
\end{proof}

\bigskip \centerline {\bf ACKNOWLEDGEMENTS}
\bigskip
This research was partially supported by NSFC (11501257,11671069,11771212), Qing Lan Project of Jiangsu Province and Jinling Institute of Technology (jit-gjfh-201502, jit-b-201615, jit-b-201638, jit-fhxm-201707).  Part of this project was carried out while the corresponding author was visiting at Capital Normal University in January 2016. He would like to thank Professor Changchang Xi for his support and suggestions.
The authors would like to thank Lidia Angeleri-H$\ddot{\rm{u}}$gel, Xiaowu Chen, Nanqing Ding, Xianhui Fu, Jan $\check{\rm{S}}$aroch, Jiaqun Wei and Haiyan Zhu for helpful discussions on parts of this article. The authors are grateful to the referee for the valuable comments.

\end{document}